\documentclass[11pt]{article}
\usepackage[T1]{fontenc}
\usepackage{lmodern}
\usepackage{microtype}
\usepackage{amsmath,amssymb,amsthm,mathtools}
\usepackage{enumitem}
\usepackage[a4paper,margin=1in]{geometry}
\usepackage{fancyhdr}
\usepackage[hidelinks]{hyperref}

\newtheorem{theorem}{Theorem}[section]
\newtheorem{proposition}[theorem]{Proposition}
\newtheorem{lemma}[theorem]{Lemma}
\newtheorem{corollary}[theorem]{Corollary}
\theoremstyle{definition}
\newtheorem{definition}[theorem]{Definition}
\newtheorem{example}[theorem]{Example}
\theoremstyle{remark}
\newtheorem{remark}[theorem]{Remark}

\newcommand{\DesR}{\operatorname{Des}_{R}}
\newcommand{\desR}{\operatorname{des}_{R}}
\newcommand{\hocolim}{\operatorname*{hocolim}}
\newcommand{\sd}{\operatorname{sd}}
\newcommand{\Ind}{\operatorname{Ind}}

\newcommand{\cat}{\operatorname{cat}}
\newcommand{\wtH}{\widetilde H}
\newcommand{\wtC}{\widetilde C}
\newcommand{\wtP}{\widetilde P}
\newcommand{\Z}{\mathbb Z}
\newcommand{\C}{\mathbb C}
\newcommand{\R}{\mathbb R}
\newcommand{\bbk}{\Bbbk}
\newcommand{\BK}{\mathcal B}

\newcommand{\XK}{X_{\mathcal K}}
\newcommand{\cW}{C_W}

\title{\textbf{Parabolic Homotopy Colimits and Coxeter Descents}}
\author{Yifan Zhang\\[1mm]
\small Department of Algebra, Faculty of Mathematics and Physics, Charles University,\\
\small Sokolovsk\'a 49/83, 186 75 Praha 8, Czech Republic\\
\small Department of Applied Mathematics,\\
\small Faculty of Electrical Engineering and Computer Science, VSB--Technical University of Ostrava,\\
\small 17. listopadu 2172/15, 708 00 Ostrava-Poruba, Czech Republic\\
\small Department of Mathematics, Faculty of Science, University of Ostrava,\\
\small 30. dubna 22, 702 00 Ostrava, Czech Republic\\
\small \texttt{yifan.zhang@osu.cz}}
\date{}

\begin{document}
\maketitle

\begin{abstract}
Let $G$ be a compact, connected, simply connected semisimple Lie group with Weyl group $W$ and simple reflections $S$. For a simplicial complex $\mathcal K$ on $S$, form the homotopy colimit $X_{\mathcal K}(G)=\operatorname*{hocolim}_{I\in\mathcal K}G/G_I$ of standard partial flag manifolds. We compute its integral homology. If $\operatorname{Des}_R(w)$ is the right descent set of $w\in W$ and $\ell(w)$ its Coxeter length, then
$$
H_n(X_{\mathcal K}(G);\mathbb Z)\cong
\bigoplus_{w\in W}\widetilde H_{n-2\ell(w)-1}(\mathcal K_{\operatorname{Des}_R(w)};\mathbb Z).
$$
Thus the induced subcomplexes of $\mathcal K$ supply the topological data, while the Weyl group determines which subcomplex occurs and the Schubert-degree shift. The proof gives a chain-level splitting and an integral Morse reduction. We derive homotopy detection, duality and rigidity results, and recover polyhedral products, matroid--Tutte formulas, and the adjoint sphere as special cases.
\end{abstract}


\section{Introduction}
Partial flag manifolds come with many natural projection maps, and homotopy colimits provide a way to glue these spaces together.  The purpose of this paper is to compute the topology of a family of such gluings and to show that the answer is controlled by a simple combination of two kinds of data: induced subcomplexes of a simplicial complex and descent sets in a Weyl group.

Let $G$ be a compact, connected, simply connected semisimple Lie group.  Choose a maximal torus and a set $S$ of simple reflections for its Weyl group $W$.  For each subset $I\subseteq S$ there is a standard parabolic subgroup $G_I\subseteq G$ and a partial flag manifold $G/G_I$.  If $I\subset J$, there is a canonical projection $G/G_I\to G/G_J$. Now let $\mathcal K$ be a simplicial complex on the vertex set $S$. Keeping only those partial flag manifolds for which $I$ is a face of $\mathcal K$ gives the diagram $I\mapsto G/G_I$ for $I\in\mathcal K$, and we study its homotopy colimit
\begin{equation}\label{eq:introXK}
  X_{\mathcal K}(G):=\hocolim_{I\in\mathcal K}G/G_I.
\end{equation}
The basic question is:

\begin{quote}
How does the topology of $X_{\mathcal K}(G)$ depend on the combinatorics of $\mathcal K$ and on the Weyl group of $G$?
\end{quote}

This family contains several familiar spaces.  If $\mathcal K=2^J$ is a simplex, then $J$ is terminal in the indexing category and $X_{2^J}(G)\simeq G/G_J$.
If $G$ is simple and $\mathcal K=\partial\Delta^S$ is the boundary of the simplex on $S$, Castellana and Kitchloo identify the same homotopy-colimit model with the unit adjoint sphere $S(\mathfrak g)$ \cite[Theorem~1.1]{CK}.  At the opposite end of the Lie-theoretic spectrum, when $G=(SU(2))^r$, the spaces $X_{\mathcal K}(G)$ are the familiar polyhedral products $(D^3,S^2)^{\mathcal K}$ up to homotopy; see Theorem~\ref{thm:polyprod}.  Thus \eqref{eq:introXK} interpolates between partial flag manifolds, polyhedral products, and the proper-parabolic adjoint-sphere construction.

\paragraph{The main result in one sentence.}
For each $w\in W$, its right descent set $\DesR(w)=\{s\in S:\ell(ws)<\ell(w)\}$ selects the induced subcomplex $\mathcal K_{\DesR(w)}$, and the homology of this induced subcomplex contributes to $H_*(X_{\mathcal K}(G))$ after a degree shift by the real dimension $2\ell(w)$ of the Schubert cell indexed by $w$.

Here $\mathcal K_D$ denotes the induced subcomplex on $D\subseteq S$.  The precise statement is as follows; the basic Coxeter terminology is reviewed in Section~\ref{sec:background}.

\begin{theorem}[Main theorem: descent--Hochster decomposition]\label{thm:intro}
Let $G$ be as above, with Weyl system $(W,S)$, and let $\mathcal K\subseteq2^S$ be a simplicial complex.  Then
\begin{equation}\label{eq:intromain}
  H_n(X_{\mathcal K}(G);\Z)
  \cong
  \bigoplus_{w\in W}
  \wtH_{n-2\ell(w)-1}
  (\mathcal K_{\DesR(w)};\Z),
\end{equation}
where augmented reduced homology is used, with
$\wtH_{-1}(\{\varnothing\};\Z)=\Z$.

More precisely, the cellular chain complex of $X_{\mathcal K}(G)$ splits by Weyl-group element, and the summand indexed by $w$ admits an explicit integral Morse reduction to $\Sigma^{2\ell(w)+1}\wtC_*(\sd\mathcal K_{\DesR(w)};\Z)$. The same chain-level statement is defined for every finite Coxeter system, whether or not it comes from a Lie group.
\end{theorem}

Formula \eqref{eq:intromain} separates the two sources of information cleanly.  The simplicial complex $\mathcal K$ supplies the groups $\wtH_*(\mathcal K_D)$; the Weyl group tells us how often each $D$ occurs as a descent set and in which degrees it appears.  Every subset $D\subseteq S$ occurs: if $w_D$ is the longest element of the standard parabolic subgroup $W_D$, then $\DesR(w_D)=D$ \cite[Section~2.4]{BB}.  Consequently the homology of every induced subcomplex of $\mathcal K$ is visible inside the homology of $X_{\mathcal K}(G)$.

This has several consequences.  Over a field, Hochster's formula identifies the groups $\wtH_*(\mathcal K_D)$ with the squarefree multigraded Betti pieces of the Stanley--Reisner ring, so \eqref{eq:intromain} gives a Weyl-descent regrading of the Hochster data.  Torsion in an induced subcomplex produces torsion in $X_{\mathcal K}(G)$.  For inclusions $\mathcal K\subseteq\mathcal L$, the induced map $X_{\mathcal K}(G)\to X_{\mathcal L}(G)$ is a homotopy equivalence exactly when all induced-subcomplex maps $\mathcal K_D\to\mathcal L_D$ are homology equivalences.  If $G$ is simple, the boundary simplex is characterized among proper complexes on $S$ by the condition that $X_{\mathcal K}(G)$ be a homology sphere.  Matroid independence complexes give a Tutte-polynomial specialization.

The point of the chain-level argument is that these statements do not come from a spectral-sequence rank calculation.  The Schubert cells of the partial flag manifolds are compatible with the parabolic projections.  After passing to the simplicial-replacement model of the homotopy colimit, the differential never changes the Weyl label $w$.  The $w$-labelled piece is a relative order-complex chain complex, and the elementary closure map $I\mapsto I\cap\DesR(w)$ reduces it to the induced subcomplex $\mathcal K_{\DesR(w)}$.  This is the mechanism behind the main theorem.

\paragraph{Guide to the paper.}
Section~\ref{sec:background} introduces the terminology needed from Coxeter theory, simplicial complexes, partial flag manifolds, and homotopy colimits, and works out the rank-two example $G=SU(3)$.  Sections~\ref{sec:bar} and \ref{sec:schubert} construct the cellular chain model.  Section~\ref{sec:hochster} proves Theorem~\ref{thm:intro} and gives examples, including a non-Boolean type-$A_3$ calculation and an example with torsion.  Section~\ref{sec:morse} gives the integral Morse refinement.  Section~\ref{sec:duality} develops the homotopy-detection, Alexander-symmetry, and sphere-rigidity consequences.  Section~\ref{sec:specializations} treats polyhedral products and matroids, while Section~\ref{sec:bivector} returns to the proper-parabolic adjoint sphere and its bivector interpretation.  The descent--Stirling cell enumeration is recorded separately in Appendix~\ref{sec:enumerator}.

\paragraph{Relation with earlier constructions.}
Face-category diagrams associated with simplicial complexes are classical in toric topology.  Panov--Ray--Vogt study colimit and homotopy-colimit constructions underlying Stanley--Reisner and Davis--Januszkiewicz spaces \cite{PRV}, while Bahri--Bendersky--Cohen--Gitler obtain induced-subcomplex decompositions for polyhedral products \cite{BBCG}.  Limonchenko--Solomadin describe quotients of moment-angle complexes by homotopy colimits of toric diagrams \cite{LS}.  Adams--Reiner give a colorful refinement of Hochster's formula attached to a proper vertex coloring \cite{AR}.  In the present construction the induced subcomplex is selected instead by a Weyl-group descent set and carries the Schubert shift $2\ell(w)$.

Parabolic homotopy decompositions also occur in Kac--Moody topology \cite{Kitchloo}, and homotopy colimits of homogeneous spaces arise in the topology of commuting elements in compact Lie groups \cite{KT}.  Cho--Choi--Kaji express the homology of real toric spaces through induced subcomplexes and, for Weyl-chamber examples of types $A$ and $B$, relate the resulting representations to descent combinatorics \cite{CCK}.  Douvropoulos and Josuat--Verg\`es study the top homology representation of a parabolic coset poset \cite{DJV}.  These constructions differ from the diagram $I\mapsto G/G_I$ considered here.  We have not found in the cited literature the fixed-$w$ chain decomposition of the normalized parabolic homotopy-colimit differential or its reduction to $\mathcal K_{\DesR(w)}$.

\section{Background and a running example}\label{sec:background}
This section collects the conventions used in the main theorem.  Readers familiar with Weyl groups and homotopy colimits may skip directly to Section~\ref{sec:bar}.

\subsection{Weyl groups, Coxeter length, and descents}
A finite Coxeter system $(W,S)$ consists of a finite group $W$ generated by a finite set $S$ of involutions, with relations $(st)^{m_{st}}=e$ for suitable integers $m_{st}\ge2$ when $s\ne t$; see \cite[Chapters~1--2]{BB}. The length $\ell(w)$ is the smallest number of generators needed to write $w$; an expression of this length is called \emph{reduced}. A simple reflection $s\in S$ is a \emph{right descent} of $w$ if multiplying by $s$ on the right shortens the word, and we write $\DesR(w)=\{s\in S:\ell(ws)<\ell(w)\}$ and $\desR(w)=|\DesR(w)|$.
For a compact connected semisimple Lie group, $W$ is its Weyl group and $S$ is obtained by choosing a system of simple roots.  A reader interested only in the Lie-group case may take this as the source of $(W,S)$ throughout the paper.

The smallest nontrivial example is type $A_2$, where $W=S_3$ and $S=\{s_1,s_2\}$ with $s_1=(12)$ and $s_2=(23)$.  The length and right descent sets are
\[
\begin{array}{c|c|c}
 w & \ell(w) & \DesR(w)\\ \hline
 e & 0 & \varnothing\\
 s_1 & 1 & \{s_1\}\\
 s_2 & 1 & \{s_2\}\\
 s_1s_2 & 2 & \{s_2\}\\
 s_2s_1 & 2 & \{s_1\}\\
 w_0=s_1s_2s_1 & 3 & \{s_1,s_2\}.
\end{array}
\]
Thus a descent set records exactly which simple reflections can shorten a given Weyl-group element. In type $A$, if a permutation is written in one-line notation, $s_i$ is a right descent precisely when $w(i)>w(i+1)$; thus $\DesR(w)$ is the usual permutation descent set.

For $I\subseteq S$, let $W_I$ be the subgroup generated by $I$. The set $W^I=\{w\in W:\ell(ws)>\ell(w)\text{ for every }s\in I\}$ consists of the minimal-length representatives of the right cosets $W/W_I$. Equivalently,
\begin{equation}\label{eq:backgroundcriterion}
  w\in W^I \quad\Longleftrightarrow\quad I\cap\DesR(w)=\varnothing.
\end{equation}
This criterion is the point at which descent sets enter the topology.

We record the other standard Coxeter facts used later. Since $W$ is finite, every standard parabolic subgroup $W_I$ has a unique longest element, denoted $w_I$, and $\DesR(w_I)=I$; thus every subset of $S$ occurs as a right descent set \cite[Section~2.4]{BB}. In particular $w_0:=w_S$ is the longest element of $W$, $N:=\ell(w_0)$, and $e$ and $w_0$ are respectively the unique elements with right descent sets $\varnothing$ and $S$. The longest-element identities
\begin{equation}\label{eq:longest}
  \ell(w_0w)=N-\ell(w),\qquad \DesR(w_0w)=S\setminus\DesR(w)
\end{equation}
hold for every $w\in W$ \cite[Section~1.8]{Humphreys}. Finally, if $I\subset J\subseteq S$, put $W_J^I:=\{v\in W_J:\ell(vs)>\ell(v)\text{ for every }s\in I\}$. Every $w\in W^I$ has a unique parabolic factorization $w=uv$ with $u\in W^J$, $v\in W_J^I$, and $\ell(w)=\ell(u)+\ell(v)$ \cite[Section~2.4]{BB}. These are the only Coxeter-theoretic facts about descents and parabolic quotients used in the proofs.

\subsection{Simplicial complexes and induced subcomplexes}
A simplicial complex $\mathcal K$ on the ground set $S$ is a collection of subsets of $S$ that contains $\varnothing$ and is closed under taking subsets.  We allow a vertex of $S$ not to occur as a one-element face; such a vertex is called a ghost vertex.  For $D\subseteq S$, the \emph{induced subcomplex} on $D$ is $\mathcal K_D=\{I\in\mathcal K:I\subseteq D\}$. For example, if $S=\{1,2,3\}$ and $\mathcal K$ is the path with edges $\{1,2\}$ and $\{2,3\}$, then $\mathcal K_{\{1,3\}}$ consists of two isolated vertices, so $\wtH_0(\mathcal K_{\{1,3\}};\Z)\cong\Z$.
This is exactly the induced subcomplex that produces the nontrivial homology in Example~\ref{ex:a3}.

Two complexes will appear repeatedly. The full simplex $2^S$ contains every subset of $S$, while its boundary $\partial\Delta^S=\{I\subsetneq S\}$ contains every proper subset but omits $S$ itself.

\subsection{Partial flag manifolds and the homotopy colimit}
Let $G$ be compact, connected, simply connected, and semisimple, with Weyl system $(W,S)$.  For $I\subseteq S$, the standard subgroup $G_I$ gives a partial flag manifold $G/G_I$.  The extreme cases are $G/G_{\varnothing}=G/T$ and $G/G_S=\mathrm{pt}$, where $T$ is a maximal torus.  If $I\subset J$, there is a canonical projection $G/G_I\to G/G_J$.

For a simplicial complex $\mathcal K\subseteq2^S$, the space $X_{\mathcal K}(G)$ is the homotopy colimit of these projections over the face poset of $\mathcal K$.  We use the Bousfield--Kan simplicial-replacement model \cite[Chapter~XII]{BK}.  Concretely, for every strict chain of faces $I_0<I_1<\cdots<I_q$ one takes a copy of $G/G_{I_0}\times\Delta^q$ and glues its faces using the parabolic projections and the usual face maps of the simplex.  This description is useful because a Schubert cell in $G/G_{I_0}$, crossed with the open simplex $\mathring\Delta^q$, becomes a cell of the homotopy colimit; Section~\ref{sec:schubert} makes this precise.

The two simplest indexing complexes already recover familiar spaces. If $\mathcal K=2^J$, the face $J$ is terminal and $X_{2^J}(G)\simeq G/G_J$.
If $G$ is simple and $\mathcal K=\partial\Delta^S$, Castellana--Kitchloo identify the resulting homotopy colimit with $S(\mathfrak g)$ \cite[Theorem~1.1]{CK}.

\subsection{A rank-two example}
Let $G=SU(3)$, whose Weyl group is the type-$A_2$ group $S_3$ listed above.  Take $\mathcal K=\partial\Delta^{\{s_1,s_2\}}=\{\varnothing,\{s_1\},\{s_2\}\}$.
For a proper nonempty subset $D\subsetneq S$, the induced subcomplex $\mathcal K_D$ is a point and has zero reduced homology.  The full induced subcomplex $\mathcal K_S$ consists of two points, so $\wtH_0(\mathcal K_S;\Z)\cong\Z$. Among the six Weyl-group elements, only the longest element $w_0$ has descent set $S$, and $\ell(w_0)=3$. Formula \eqref{eq:intromain} therefore gives one nonzero positive-degree class, in degree $2\ell(w_0)+0+1=7$. Together with the identity contribution in degree zero, this gives $H_*(X_{\mathcal K}(SU(3));\Z)\cong H_*(S^7;\Z)$,
in agreement with the adjoint-sphere identification.  This calculation contains the general mechanism in its simplest form: descents choose an induced subcomplex, and the Coxeter length determines where its homology appears.

\section{The chain model: parabolic bar complexes}\label{sec:bar}
The Bousfield--Kan model described in Section~\ref{sec:background} has one simplex direction for every chain of faces of $\mathcal K$.  We now record the corresponding chain complex abstractly.  The definition is given for an arbitrary finite Coxeter system; Section~\ref{sec:schubert} explains why it is the cellular chain complex of $X_{\mathcal K}(G)$ in the Weyl-group case.
Let $(W,S)$ be a finite Coxeter system and let $\mathcal K\subseteq 2^S$ be an abstract simplicial complex. We require $\varnothing\in\mathcal K$ but allow ghost vertices. We use the parabolic subgroups $W_I$ and minimal representatives $W^I$ introduced in Section~\ref{sec:background}; in particular, \eqref{eq:backgroundcriterion} will be used repeatedly. For every $I\in\mathcal K$ define $\cW(I)=\bigoplus_{w\in W^I}\Z e_w^I$, with $|e_w^I|=2\ell(w)$ and zero internal differential. If $I\subset J$ are faces of $\mathcal K$, define
\begin{equation}\label{eq:parabolicmap}
  p_{I,J}(e_w^I)=
  \begin{cases}
    e_w^J,&w\in W^J,\\
    0,&w\notin W^J.
  \end{cases}
\end{equation}
These maps are functorial.

\begin{definition}[Parabolic bar complex]
The parabolic bar complex over $\mathcal K$, denoted $\BK_{\mathcal K}(W)$, is the total complex of the normalized simplicial replacement of $\cW:\cat(\mathcal K)\to\mathrm{Ch}_{\Z}$. Its basis elements are $[w;I_0<\cdots<I_q]$, where $I_j\in\mathcal K$ and $w\in W^{I_0}$, and the total degree is $2\ell(w)+q$.
\end{definition}

The notation has a direct geometric meaning in the Weyl-group case.  The label $w$ denotes a Schubert cell of real dimension $2\ell(w)$ in $G/G_{I_0}$, while the chain $I_0<\cdots<I_q$ denotes a $q$-simplex in the simplicial replacement.  Thus the generator $[w;I_0<\cdots<I_q]$ corresponds to a cell of total dimension $2\ell(w)+q$.

Since the internal differentials vanish,
\begin{equation}\label{eq:bardiff}
\begin{aligned}
\partial[w;I_0<\cdots<I_q]
  &=p_{I_0,I_1}(e_w^{I_0})[I_1<\cdots<I_q]\\
  &\quad+\sum_{j=1}^q(-1)^j[w;I_0<\cdots<\widehat I_j<\cdots<I_q].
\end{aligned}
\end{equation}
The first term is zero when $w\notin W^{I_1}$. When $\mathcal K=\partial\Delta^S=\{I\subsetneq S\}$, we abbreviate $\BK(W):=\BK_{\partial\Delta^S}(W)$.

\begin{remark}
The construction is purely Coxeter-theoretic and applies to noncrystallographic finite Coxeter groups. The grading $2\ell(w)$ is chosen so that, for Weyl groups, it agrees with the real dimension of the corresponding Schubert cell.
\end{remark}

\section{Schubert cells and the topological realization}\label{sec:schubert}
We now return to the Lie-group setting and identify the abstract chain model with the cellular chains of the space $X_{\mathcal K}(G)$.  The only Lie-theoretic input is the compatibility of Schubert cells with the standard parabolic projections.
Assume now that $W$ is the Weyl group of a compact, connected, simply connected semisimple Lie group $G$. Let $G_{\C}$ be the complexification and choose a Borel subgroup $B$. For $I\subseteq S$, let $P_I\subseteq G_{\C}$ be the standard parabolic and let $G_I=G\cap P_I$; equivalently, $G_I$ is the connected maximal-rank subgroup generated by a maximal torus and the rank-one root subgroups indexed by $I$, as in \cite[Definition~3.1]{CK}. The compact homogeneous space $G/G_I$ is diffeomorphic to $G_{\C}/P_I$ and has the Schubert CW decomposition \cite{Brion}
\[
  G/G_I=\coprod_{w\in W^I}C_w^I,
  \qquad C_w^I\cong\C^{\ell(w)}.
\]
We orient every cell by its complex orientation. Since all Schubert cells have even real dimension, the cellular differential of each $G/G_I$ is zero.

\begin{proposition}[Cellular chain map of a parabolic projection]\label{prop:projection}
Let $I\subset J\subseteq S$. For $w\in W^I$, write its parabolic factorization $w=uv$ with $u\in W^J$, $v\in W_J^I$, and $\ell(w)=\ell(u)+\ell(v)$.
The projection $\pi_{I,J}:G/G_I\to G/G_J$ maps $C_w^I$ onto $C_u^J$. On cellular chains,
\[
  (\pi_{I,J})_\#(e_w^I)=
  \begin{cases}
    e_w^J,&v=e\ \text{(equivalently }w\in W^J\text{)},\\
    0,&v\ne e.
  \end{cases}
\]
Thus the cellular Schubert diagram is strictly isomorphic to the abstract diagram $\cW$ of \eqref{eq:parabolicmap}.
\end{proposition}

\begin{proof}
The standard projection sends the $B$-orbit $BwP_I/P_I$ to $BwP_J/P_J=BuP_J/P_J=C_u^J$; this is the usual parabolic factorization of Schubert cells in partial flag varieties \cite{Brion}. If $v\ne e$, then $\ell(u)<\ell(w)$, so the image of the real $2\ell(w)$-cell lies in the $(2\ell(w)-1)$-skeleton of $G/G_J$. Its degree-$2\ell(w)$ cellular chain class is therefore zero. If $v=e$, the restriction $C_w^I\to C_w^J$ is an isomorphism of complex affine spaces and has degree $+1$ for the chosen orientations.
\end{proof}

\begin{lemma}[Cellular realization of the simplicial replacement]\label{lem:cellularrealization}
Let $\mathcal C$ be a finite poset and let $F:\mathcal C\to\mathrm{Top}$ be a diagram of finite CW complexes and cellular maps. In the Bousfield--Kan simplicial-replacement model for $\hocolim_{\mathcal C}F$, the products $e^p\times\mathring\Delta^q$, indexed by $c_0<\cdots<c_q$ and with $e^p$ an open $p$-cell of $F(c_0)$, form the open cells of a CW structure of dimension $p+q$. With the product orientation, the cellular chain complex is the usual total complex of the normalized simplicial replacement of $C_*^{\mathrm{cell}}(F)$; on a product of cellular degree $p$ and simplicial degree $q$, the simplex-face part of the differential carries the Koszul sign $(-1)^p$.
\end{lemma}

\begin{proof}
Use the quotient model $\left(\coprod_{c_0<\cdots<c_q}F(c_0)\times\Delta^q\right)/\!\sim$ for the simplicial replacement; see \cite[Chapter~XII]{BK}. Choose characteristic maps $D^p\to F(c_0)$. The product $D^p\times\Delta^q$ is a $(p+q)$-ball. Its boundary is the union of $\partial D^p\times\Delta^q$ and the products with the codimension-one faces of $\Delta^q$. The first part maps to lower CW dimension in $F(c_0)$, and the second part maps, through a structure map of $F$ when the zeroth simplex face is removed, to cells of total dimension at most $p+q-1$. Thus the product characteristic maps attach inductively by total dimension. No relation in the simplicial replacement identifies two points in the interiors $e^p\times\mathring\Delta^q$.

The product boundary formula gives the cellular boundary in $F(c_0)$ together with $(-1)^p$ times the alternating simplex-face boundary. On the zeroth simplex face the coefficient is the cellular chain map of $F(c_0)\to F(c_1)$; the other simplex faces delete an entry of the chain. This is the standard total-complex differential.
\end{proof}

\begin{corollary}[Schubert--bar CW model]\label{cor:schubertbar}
For every simplicial complex $\mathcal K\subseteq2^S$, the Bousfield--Kan realization of the restricted diagram $I\mapsto G/G_I$ has a CW structure with cells $C_w^{I_0}\times\mathring\Delta^q$, indexed by $I_0<\cdots<I_q$ and $w\in W^{I_0}$, of dimension $2\ell(w)+q$. Its cellular chain complex is naturally identified with $\BK_{\mathcal K}(W)$.
\end{corollary}

\begin{proof}
Apply Lemma~\ref{lem:cellularrealization} to the Schubert diagram. By Proposition~\ref{prop:projection}, its cellular chain maps are the maps $p_{I,J}$ of \eqref{eq:parabolicmap}. The Schubert cells have even real dimension $p=2\ell(w)$, so the Koszul sign in Lemma~\ref{lem:cellularrealization} is $+1$ on every Schubert generator. Since the internal Schubert differentials vanish, the resulting cellular differential is exactly \eqref{eq:bardiff}. For $\mathcal K=\partial\Delta^S$, this is the same explicit Bousfield--Kan quotient used by Castellana--Kitchloo \cite[Lemma~3.3]{CK}.
\end{proof}

\begin{proposition}[Simple connectivity]\label{prop:simplyconnected}
For every simplicial complex $\mathcal K\subseteq2^S$, the parabolic homotopy colimit $\XK(G)$ is simply connected.
\end{proposition}

\begin{proof}
Use the Schubert--bar CW structure of Corollary~\ref{cor:schubertbar}. A cell has dimension $2\ell(w)+q$, so every $1$-cell has $w=e$ and $q=1$. More generally, the identity-labelled cells form exactly the order complex $A_{\mathcal K}$, including its $2$-skeleton. Since $A_{\mathcal K}$ is a cone, its $2$-skeleton has trivial fundamental group. The full $2$-skeleton of $\XK(G)$ is obtained from this simply connected $2$-complex by attaching additional $2$-cells, namely those with $\ell(w)=1$ and $q=0$. Attaching $2$-cells only adds relations to the fundamental group, and the inclusion of the $2$-skeleton induces an isomorphism on $\pi_1$ \cite[Section~1.2]{Hatcher}. Hence $\XK(G)$ is simply connected.
\end{proof}

\begin{remark}
Corollary~\ref{cor:schubertbar} uses one fixed choice of Schubert cells that is compatible with every structure map in the diagram; no separate cellular approximation of the parabolic projections is required.
\end{remark}

\section{Proof of the main theorem: the descent--Hochster decomposition}\label{sec:hochster}
This section proves Theorem~\ref{thm:intro}.  The first step is an exact splitting by the Weyl label $w$.  The second step identifies the relative complex in that summand with a suspension of the induced subcomplex $\mathcal K_{\DesR(w)}$.
Put
\[
  P_{\mathcal K}=\mathcal K,
  \qquad A_{\mathcal K}=\Delta(P_{\mathcal K}),
\]
where the empty face is included as a vertex of the order complex. Thus $A_{\mathcal K}$ is a cone with apex $\varnothing$. For $D\subseteq S$ put
\[
  Q_{\mathcal K,D}=\{I\in\mathcal K:I\cap D\ne\varnothing\},
  \qquad
  L_{\mathcal K,D}=\Delta(Q_{\mathcal K,D}),
\]
with $L_{\mathcal K,D}=\varnothing$ if the poset is empty, and let
\[
  \mathcal K_D=\{I\in\mathcal K:I\subseteq D\}.
\]

\subsection{Exact splitting by Coxeter labels}
\begin{theorem}[Exact descent decomposition]\label{thm:exactsplit}
There is a based chain isomorphism
\begin{equation}\label{eq:exactsplit}
  \BK_{\mathcal K}(W)
  \cong
  \bigoplus_{w\in W}
  \Sigma^{2\ell(w)}C_*(A_{\mathcal K},L_{\mathcal K,\DesR(w)};\Z).
\end{equation}
Under this isomorphism $[w;I_0<\cdots<I_q]$ corresponds to the relative simplex $[I_0<\cdots<I_q]$ in the $w$-summand.
\end{theorem}

\begin{proof}
The differential never changes $w$, so the bar complex splits into its $w$-labelled subcomplexes. By \eqref{eq:backgroundcriterion}, a basis element with label $w$ exists exactly when $I_0\cap\DesR(w)=\varnothing$. Since the chain is increasing, a simplex lies in $L_{\mathcal K,D}$ exactly when its first vertex meets $D$. Hence the relative simplices are precisely the admissible bar chains. Deleting the first vertex is zero in the relative boundary exactly when $I_1\cap D\ne\varnothing$, which agrees with the zeroth bar face by \eqref{eq:parabolicmap}. All remaining alternating faces agree term by term.
\end{proof}

\subsection{Reduction to induced subcomplexes and Hochster's formula}
\begin{lemma}[Relative-cone reduction]\label{lem:conereduction}
For every $D\subseteq S$ there is a unit acyclic matching on $C_*(A_{\mathcal K},L_{\mathcal K,D};\Z)$ whose Morse complex is canonically the suspended augmented chain complex of $L_{\mathcal K,D}$. Explicitly, write
\[
  A_{\mathcal K}=\varnothing*B_{\mathcal K},
  \qquad
  B_{\mathcal K}=\Delta(\mathcal K\setminus\{\varnothing\}).
\]
Pair every base simplex $\sigma\subset B_{\mathcal K}$ not contained in $L_{\mathcal K,D}$ with its cone $\varnothing*\sigma$.
\end{lemma}

\begin{proof}
The coefficient of $\sigma$ in $\partial(\varnothing*\sigma)$ is a unit. This is the standard cone matching and is acyclic. If $\tau\subset L_{\mathcal K,D}$, then every face of $\tau$ again lies in $L_{\mathcal K,D}$, so the boundary of $\varnothing*\tau$ never enters a matched pair outside $L_{\mathcal K,D}$. Thus no gradient-path correction occurs among the critical generators. For a vertex $v$ of $L_{\mathcal K,D}$, the relative boundary of the edge $\varnothing*v$ is $-\varnothing$, which is precisely the augmentation term.
\end{proof}

\begin{lemma}[Descent closure]\label{lem:descentclosure}
If $L_{\mathcal K,D}\ne\varnothing$, the map
\[
  c_D:Q_{\mathcal K,D}\longrightarrow Q_{\mathcal K,D},
  \qquad c_D(I)=I\cap D,
\]
is a descending closure operator. Its image is the poset of nonempty faces of $\mathcal K_D$. Consequently
\[
  L_{\mathcal K,D}\searrow \sd\mathcal K_D
\]
by a sequence of unit simplicial collapses.
\end{lemma}

\begin{proof}
Because $\mathcal K$ is closed under taking subsets, $I\cap D$ is a face of $\mathcal K$; it is nonempty for $I\in Q_{\mathcal K,D}$. The map is idempotent, order preserving, and satisfies $c_D(I)\subseteq I$. Its fixed points are exactly the nonempty faces contained in $D$. The conclusion follows from the closure-operator collapse theorem \cite{KozlovCollapse}.
\end{proof}

\begin{theorem}[Descent--Hochster reduction]\label{thm:hochster}
For every $D\subseteq S$ there is an integral chain-homotopy equivalence
\[
  C_*(A_{\mathcal K},L_{\mathcal K,D};\Z)
  \simeq
  \Sigma\wtC_*(\sd\mathcal K_D;\Z),
\]
where augmented reduced chains are used and $\wtC_{-1}(\{\varnothing\};\Z)=\Z$. Hence
\begin{equation}\label{eq:homology}
  H_n(\BK_{\mathcal K}(W);\Z)
  \cong
  \bigoplus_{w\in W}
  \wtH_{n-2\ell(w)-1}(\mathcal K_{\DesR(w)};\Z).
\end{equation}
For Weyl groups the same formula computes $H_n(\XK(G);\Z)$.
\end{theorem}

\begin{proof}
Apply Lemma~\ref{lem:conereduction} and then lift the collapse matching of Lemma~\ref{lem:descentclosure} one degree to the cone generators. If $L_{\mathcal K,D}$ is empty, the apex is the only critical generator, agreeing with the suspended augmented chains of $\{\varnothing\}$. Barycentric subdivision preserves simplicial homology. The final statement follows from Corollary~\ref{cor:schubertbar}.
\end{proof}

\begin{corollary}[Hochster regrading]\label{cor:hochsterregrading}
Assume that $W$ is a Weyl group with compact realization $G$ as in Section~\ref{sec:schubert}. Let $\bbk$ be a field and let $\bbk[\mathcal K]$ be the Stanley--Reisner ring on the ground set $S$. If $D=\DesR(w)$, then
\[
  \dim_{\bbk}H_{2\ell(w)+|D|-i}^{(w)}(\XK(G);\bbk)
  =\beta_{i,D}(\bbk[\mathcal K]),
\]
where the left-hand side denotes the $w$-labelled direct summand and $\beta_{i,D}$ is the squarefree multigraded Betti number.
\end{corollary}

\begin{proof}
Hochster's formula gives
\[
  \beta_{i,D}(\bbk[\mathcal K])
  =\dim_{\bbk}\wtH^{|D|-i-1}(\mathcal K_D;\bbk)
\]
\cite{Hochster}. Over a field homology and cohomology have the same dimension. Compare degrees with \eqref{eq:homology}.
\end{proof}

\begin{corollary}[Descent-weighted Poincar\'e series]\label{cor:poincare}
Assume that $W$ is a Weyl group with compact realization $G$ as in Section~\ref{sec:schubert}. Put
\[
  A_{W,D}(q)=\sum_{\DesR(w)=D}q^{\ell(w)},
  \qquad
  \wtP_{\mathcal K_D}(t)=\sum_{j\ge -1}\operatorname{rank}\wtH_j(\mathcal K_D;\Z)t^j.
\]
Then
\[
  P_{\XK(G)}(t)
  =t\sum_{D\subseteq S}A_{W,D}(t^2)\wtP_{\mathcal K_D}(t).
\]
The analogous equality of dimensions holds over any field.
\end{corollary}

\subsection{Examples and the boundary-simplex case}
\begin{example}[A non-Boolean type $A_3$ calculation]\label{ex:a3}
Let $W=S_4$ with simple reflections $s_1,s_2,s_3$, and let $\mathcal K$ be the path
\[
  \{s_1,s_2\}\cup\{s_2,s_3\}
\]
with all of its faces. Every nonempty induced subcomplex is contractible except $\mathcal K_{\{s_1,s_3\}}$, which consists of two isolated vertices and has $\wtH_0\cong\Z$. The permutations with exact descent set $\{1,3\}$ are
\[
  2143,\quad3142,\quad3241,\quad4132,\quad4231,
\]
with inversion numbers $2,3,4,4,5$. Hence
\[
  P_{X_{\mathcal K}(SU(4))}(t)=1+t^5+t^7+2t^9+t^{11}.
\]
This example is neither a partial flag manifold nor the boundary-simplex adjoint sphere.
\end{example}

\begin{example}[Torsion created by the indexing complex]
Assume $|S|=6$ and let $\mathcal K$ be a six-vertex triangulation of $\mathbb{RP}^2$. Since $\DesR(w_0)=S$ and $\wtH_1(\mathcal K;\Z)\cong\Z/2$, Theorem~\ref{thm:hochster} gives a direct summand
\[
  \Z/2\subseteq H_{2\ell(w_0)+2}(\XK(G);\Z).
\]
Thus torsion can appear in the parabolic homotopy colimit even though the Schubert cellular complex of every individual partial flag manifold $G/G_I$ is free and concentrated in even degrees.
\end{example}

If $\mathcal K=2^J$ is a simplex, then $\mathcal K_D$ is contractible when $D\cap J\ne\varnothing$ and equals $\{\varnothing\}$ otherwise. Formula \eqref{eq:homology} leaves exactly the elements $w\in W^J$ in degrees $2\ell(w)$, recovering the Schubert homology of $G/G_J$.

For the boundary simplex put
\[
  P_S=\{I\subsetneq S\},
  \qquad K_S=\Delta(P_S),
  \qquad L_D=\Delta\{I\subsetneq S:I\cap D\ne\varnothing\}.
\]
Then Theorem~\ref{thm:exactsplit} specializes to
\begin{equation}\label{eq:boundarysplit}
  \BK(W)
  \cong
  \bigoplus_{w\in W}\Sigma^{2\ell(w)}C_*(K_S,L_{\DesR(w)};\Z).
\end{equation}
Every proper nonempty induced subcomplex of $\partial\Delta^S$ is a simplex, whereas the full induced subcomplex is $\partial\Delta^S$. Hence only $e$ and $w_0$ contribute homology.

\begin{corollary}[Boundary-simplex homology]\label{cor:boundaryhomology}
Let $r=|S|$, let $w_0$ be the longest element, and put $N=\ell(w_0)$. Then
\[
  H_k(\BK(W);\Z)\cong
  \begin{cases}
    \Z,&k=0,\\
    \Z,&k=2N+r-1,\\
    0,&\text{otherwise}.
  \end{cases}
\]
\end{corollary}

\begin{corollary}[An explicit top cycle]\label{cor:topcycle}
Fix an ordering $S=\{s_1,\ldots,s_r\}$. For $\sigma\in S_r$ put
\[
  I_j^\sigma=\{s_{\sigma(1)},\ldots,s_{\sigma(j)}\},
  \qquad1\le j\le r-1.
\]
Then
\begin{equation}\label{eq:topcycle}
  \Omega_W=\sum_{\sigma\in S_r}\operatorname{sgn}(\sigma)
  [w_0;\varnothing<I_1^\sigma<\cdots<I_{r-1}^\sigma]
\end{equation}
is a cycle in degree $2N+r-1$ and generates the top homology up to sign.
\end{corollary}

\begin{proof}
The $w_0$-summand is the cone pair $\Sigma^{2N}C_*(K_S,L_S)$ with $L_S=\sd\partial\Delta^{r-1}$. The displayed chain is the cone on the standard oriented barycentric fundamental cycle.
\end{proof}

\section{Integral algebraic Morse reduction}\label{sec:morse}
We keep the natural bases. A unit matching on a based free chain complex means an acyclic matching in which every matched differential coefficient is $\pm1$ \cite{KozlovMorse,JW}.

\begin{theorem}[Morse transfer from induced subcomplexes]\label{thm:morsetransfer}
For each $D\subseteq S$, the unit matchings in Lemmas~\ref{lem:conereduction} and \ref{lem:descentclosure} give an explicit algebraic Morse reduction of the $D$-relative pair to the suspended augmented chains of $\sd\mathcal K_D$. Consequently any unit acyclic matching on $\sd\mathcal K_D$ may be applied as a second algebraic Morse reduction, shifted one degree upward. If $\mathcal K_D\ne\{\varnothing\}$ and the chosen matching has a critical vertex, the augmentation generator can be cancelled with the suspension of one such vertex; the remaining critical generators represent the shifted reduced Morse complex of $\sd\mathcal K_D$.
\end{theorem}

\begin{proof}
The first two reductions are those of Lemmas~\ref{lem:conereduction} and \ref{lem:descentclosure}; together they identify the relative pair with the suspended augmented chain complex of $\sd\mathcal K_D$. Algebraic discrete Morse reduction can then be applied to this based complex \cite{JW}. Suspending a matched pair raises both degrees by one and preserves its unit coefficient. If a critical vertex is chosen, the suspended augmented differential sends its suspension to the augmentation generator with coefficient $\pm1$, so that pair can be cancelled as a final unit Morse step. What remains is the suspension of the reduced Morse complex of the chosen matching on $\sd\mathcal K_D$.
\end{proof}

\begin{lemma}[Relative Boolean matchings]\label{lem:booleanmatching}
After choosing a total ordering of $S$, the relative complex $C_*(K_S,L_D;\Z)$ admits a unit acyclic matching with
\begin{itemize}[nosep]
  \item one critical $0$-simplex if $D=\varnothing$;
  \item no critical simplices if $\varnothing\subsetneq D\subsetneq S$;
  \item one critical $(r-1)$-simplex if $D=S$.
\end{itemize}
\end{lemma}

\begin{proof}
For $D=\varnothing$, $K_S$ is the cone with apex $\varnothing$. Pair every simplex of the base with its cone over $\varnothing$; only the apex is critical.

For $\varnothing\subsetneq D\subsetneq S$, the induced complex $(\partial\Delta^S)_D$ is a simplex. Apply the cone reduction of Lemma~\ref{lem:conereduction} and the closure reduction of Lemma~\ref{lem:descentclosure}; the relative complex becomes the suspended augmented chains of a simplex. Use the standard cone matching on that simplex, with one critical vertex, and pair the relative apex with the cone on that vertex. No critical relative simplex remains.

For $D=S$ and $r=1$, the relative complex is the one-vertex complex $C_*(K_S,\varnothing)$, so the statement is immediate. Assume $r\ge2$. Then the base $L_S=\sd\partial\Delta^{r-1}$ is a shellable simplicial sphere and admits a perfect discrete Morse matching with one critical vertex and one critical $(r-2)$-simplex; this is the standard shelling-to-Morse construction \cite{Chari}. Lift every matched pair to the cone cells containing the apex $\varnothing$. Pair the relative apex itself with the cone on the critical vertex. The only remaining relative critical cell is the cone on the critical $(r-2)$-simplex, of dimension $r-1$. Acyclicity is inherited from the base matching, and the additional apex pair cannot create a directed cycle.
\end{proof}

\begin{theorem}[Perfect reduction of the parabolic bar complex]\label{thm:perfect}
Let $(W,S)$ be a finite Coxeter system, $r=|S|$, and $N=\ell(w_0)$. The based complex $\BK(W)$ admits a unit acyclic matching whose only critical generators are
\[
  [e;\varnothing]\quad\text{in degree }0,
  \qquad
  \xi_{w_0}\quad\text{in degree }2N+r-1.
\]
Consequently its algebraic Morse complex is
\begin{equation}\label{eq:morse}
  \BK(W)_{\mathrm{Morse}}\cong\Z[0]\oplus\Z[2N+r-1]
\end{equation}
with zero differential.
\end{theorem}

\begin{proof}
Apply Lemma~\ref{lem:booleanmatching} separately to every direct summand in \eqref{eq:boundarysplit}. Direct sums of acyclic matchings on disjoint based subcomplexes remain acyclic. In a finite Coxeter group, $e$ is the unique element with empty right descent set and $w_0$ is the unique element with right descent set $S$ \cite[Section~2.4]{BB}. Hence the $e$-summand contributes the unique critical degree-$0$ generator, all elements with nonempty proper descent set contribute none, and the $w_0$-summand contributes the unique critical Boolean degree $r-1$, shifted by $2N$.
\end{proof}

\begin{remark}
The matching is block diagonal with respect to the decomposition by $w$. Thus the cancellations can be performed independently in the $w$-summands, and their combinatorics depends only on $\DesR(w)$.
\end{remark}

\begin{remark}
The Schubert CW structures need not be regular because Schubert-variety closures may be singular. Theorem~\ref{thm:perfect} is therefore a statement of algebraic discrete Morse theory on cellular chains, not a sequence of Forman collapses of the full topological bar CW complex. A regular refinement on which the same matching becomes geometric would be a separate strengthening.
\end{remark}

\section{Topological consequences: detection, duality, and rigidity}\label{sec:duality}
Formula \eqref{eq:homology} is natural in the indexing complex.  Together with simple connectivity, this turns the additive decomposition into a homotopy-detection statement.  When $\mathcal K$ is a homology sphere, the involution $w\mapsto w_0w$ matches complementary descent sets and yields an Alexander-duality symmetry.  The same formula also characterizes the boundary simplex by the sphere property.
\subsection{Functoriality and homotopy detection}
\begin{proposition}[Functoriality in the indexing complex]\label{prop:functoriality}
Let $\mathcal K\subseteq\mathcal L\subseteq2^S$. The induced map
\[
  \BK_{\mathcal K}(W)\longrightarrow\BK_{\mathcal L}(W)
\]
respects the decomposition by $w$. On the $w$-summand it is the relative chain map
\[
  C_*(A_{\mathcal K},L_{\mathcal K,\DesR(w)})
  \longrightarrow
  C_*(A_{\mathcal L},L_{\mathcal L,\DesR(w)})
\]
induced by inclusion of face posets. In the Weyl-group case this is the cellular chain map induced by $\XK(G)\to X_{\mathcal L}(G)$.
\end{proposition}

\begin{proof}
The map of bar complexes sends a generator $[w;I_0<\cdots<I_q]$ to the generator with the same label and face chain, now regarded in $\mathcal L$. Hence it preserves the $w$-label and, under \eqref{eq:exactsplit}, is exactly the map of relative order-complex chains induced by inclusion. The topological statement follows from naturality of the simplicial replacement.
\end{proof}

\begin{theorem}[Homotopy detection under inclusions]\label{thm:homotopydetection}
Let $G$ be compact, connected, simply connected, and semisimple, and let $\mathcal K\subseteq\mathcal L\subseteq2^S$. The natural map
\[
  \XK(G)\longrightarrow X_{\mathcal L}(G)
\]
is a homotopy equivalence if and only if, for every $D\subseteq S$, the inclusion
\[
  \mathcal K_D\longrightarrow\mathcal L_D
\]
induces isomorphisms on all augmented reduced integral homology groups, including degree $-1$.
\end{theorem}

\begin{proof}
By Proposition~\ref{prop:functoriality}, the cellular map is block diagonal with respect to the decomposition by $w$. For $D=\DesR(w)$, naturality of the long exact sequence of the cone pair gives
\[
  H_q(A_{\mathcal K},L_{\mathcal K,D})\cong\wtH_{q-1}(L_{\mathcal K,D}),
  \qquad
  H_q(A_{\mathcal L},L_{\mathcal L,D})\cong\wtH_{q-1}(L_{\mathcal L,D}).
\]
The closure operators $I\mapsto I\cap D$ commute with the inclusion of face posets, so the closure-operator equivalences identify this relative map with the map induced by $\mathcal K_D\hookrightarrow\mathcal L_D$, including the augmented degree $-1$ case. If all induced-subcomplex inclusions are homology equivalences in this augmented sense, then the map of total spaces is an integral homology equivalence. Both spaces are simply connected by Proposition~\ref{prop:simplyconnected}; the homological Whitehead theorem therefore makes it a homotopy equivalence \cite[Corollary~4.33]{Hatcher}.

Conversely, suppose the map of total spaces is a homotopy equivalence. Its block-diagonal map on integral homology is then an isomorphism on every $w$-summand. Given $D\subseteq S$, take the longest element $w_D$ of $W_D$, for which $\DesR(w_D)=D$ \cite[Section~2.4]{BB}. The $w_D$-block is precisely the shifted map on augmented reduced homology of $\mathcal K_D\hookrightarrow\mathcal L_D$.
\end{proof}

\begin{corollary}[Detection of induced-subcomplex homology]\label{cor:detection}
For every $D\subseteq S$, let $w_D$ be the longest element of $W_D$. Then $\wtH_j(\mathcal K_D;\Z)$ occurs as a direct summand of
\[
  H_{2\ell(w_D)+j+1}(\BK_{\mathcal K}(W);\Z).
\]
Consequently $H_*(\BK_{\mathcal K}(W);\Z)$ is torsion free if and only if every induced subcomplex $\mathcal K_D$ has torsion-free reduced integral homology. If, in addition, $\mathcal K$ has vertex set exactly $S$, then $\XK(G)$ is contractible if and only if $\mathcal K=2^S$.
\end{corollary}

\begin{proof}
The equality $\DesR(w_D)=D$ and \eqref{eq:homology} give the first assertion. The torsion criterion follows because the full homology is a direct sum of shifted induced-subcomplex homology groups. For the last statement, the full simplex has terminal face $S$ and hence $X_{2^S}(G)\simeq G/G_S=\mathrm{pt}$. Conversely, if $\mathcal K\ne2^S$ and there are no ghost vertices, choose a minimal nonface $F$. Then $|F|\ge2$ and $\mathcal K_F=\partial\Delta^F$, whose reduced homology is nonzero; the first assertion detects it in positive degree. Thus the total space is not contractible. Equivalently, if all positive-degree homology vanishes, Proposition~\ref{prop:simplyconnected} and the homological Whitehead theorem imply contractibility \cite[Corollary~4.33]{Hatcher}.
\end{proof}

Recall from \eqref{eq:longest} that $N=\ell(w_0)$ and that left multiplication by $w_0$ complements right descent sets while reversing Coxeter length.

\begin{theorem}[Coxeter--Alexander symmetry]\label{thm:alexander}
Assume that $W$ is a Weyl group with compact realization $G$ as in Section~\ref{sec:schubert}. Let $\bbk$ be a field and suppose $\mathcal K$ is a $d$-dimensional generalized homology sphere over $\bbk$ with vertex set $S$. Put
\[
  m=2N+d+1.
\]
Then
\[
  \dim_{\bbk}H_j(\XK(G);\bbk)
  =\dim_{\bbk}H_{m-j}(\XK(G);\bbk).
\]
More precisely, after Alexander duality and the universal coefficient isomorphism over $\bbk$, the $w$-summand is dual as a finite-dimensional $\bbk$-vector space to the complementary-degree $w_0w$-summand.
\end{theorem}

\begin{proof}
For $D\subseteq S$, Alexander duality for generalized homology spheres gives
\[
  \wtH_i(\mathcal K_D;\bbk)
  \cong
  \wtH^{d-i-1}(\mathcal K_{S\setminus D};\bbk)
\]
\cite[Theorem~3.4]{FanWang}; see also \cite{GS}. Over a field, the latter group is naturally the dual vector space of $\wtH_{d-i-1}(\mathcal K_{S\setminus D};\bbk)$. A class in the $w$-summand arising from degree $i$ of $\mathcal K_D$ occurs in total degree $2\ell(w)+i+1$. By \eqref{eq:longest}, the corresponding dual vector space in the $w_0w$-summand occurs in degree
\[
  2(N-\ell(w))+(d-i-1)+1
  =m-(2\ell(w)+i+1).
\]
Summing dimensions over $w$ gives the stated symmetry.
\end{proof}

\begin{remark}
Theorem~\ref{thm:alexander} is an additive duality statement. It does not assert that $\XK(G)$ is a Poincar\'e duality space or that the vector-space duality is induced by cup or cap product.
\end{remark}

\begin{theorem}[Sphere rigidity]\label{thm:rigidity}
Let $G$ be compact, connected, simply connected, and simple, with irreducible Weyl system $(W,S)$ of rank at least two. Let $\mathcal K$ be a proper simplicial complex with vertex set exactly $S$. The following are equivalent:
\begin{enumerate}[label=(\roman*)]
  \item $\XK(G)$ is an integral homology sphere;
  \item $\XK(G)$ is homotopy equivalent to a sphere;
  \item $\mathcal K=\partial\Delta^S$.
\end{enumerate}
For (iii), the space is in fact the unit adjoint sphere.
\end{theorem}

\begin{proof}
The implication (ii)$\Rightarrow$(i) is immediate. Suppose (i) holds, so $\XK(G)$ has the integral homology of $S^m$ for some $m$. The space is connected, Proposition~\ref{prop:simplyconnected} rules out $m=1$, and $m=0$ is impossible for a connected homology sphere; hence $m\ge2$.

We first show that $\XK(G)$ is $(m-1)$-connected. If not, let $n\ge2$ be the least integer with $\pi_n(\XK(G))\ne0$. By minimality, the space is $(n-1)$-connected, so the Hurewicz theorem gives an isomorphism $\pi_n\cong H_n$. If $n<m$, this contradicts the vanishing of $H_n$. Thus no such $n<m$ exists, and the space is $(m-1)$-connected. Hurewicz in degree $m$ then gives $\pi_m\cong H_m\cong\Z$. A generator is represented by a map $S^m\to\XK(G)$ which is an integral homology equivalence; the homological Whitehead theorem makes it a homotopy equivalence \cite[Theorem~4.32 and Corollary~4.33]{Hatcher}. This proves (i)$\Rightarrow$(ii).

It remains to characterize when (i) holds. Because $\mathcal K$ is proper, it has a minimal nonface $F$. The no-ghost hypothesis gives $|F|\ge2$. Minimality implies
\[
  \mathcal K_F=\partial\Delta^F,
  \qquad
  \wtH_{|F|-2}(\mathcal K_F;\Z)\cong\Z.
\]
The longest element $w_F$ of $W_F$ satisfies $\DesR(w_F)=F$ \cite[Section~2.4]{BB}. By \eqref{eq:longest}, $w_0w_{S\setminus F}$ also has descent set $F$. If $F\subsetneq S$, these two elements are distinct. Indeed, the longest element of a finite standard parabolic subgroup is an involution \cite[Section~2.4]{BB}, so equality would give $w_0=w_Fw_{S\setminus F}$ and hence
\[
  N\le\ell(w_F)+\ell(w_{S\setminus F}).
\]
For a Weyl group, however,
\[
  N=|\Phi^+|,
  \qquad
  \ell(w_J)=|\Phi_J^+|
\]
\cite[Chapter~1]{Humphreys}. Since the root system is irreducible, the highest root has support $S$ \cite[Chapter~1]{Humphreys}. Because both $F$ and $S\setminus F$ are nonempty, this gives a positive root lying in neither parabolic subsystem, and therefore
\[
  |\Phi^+|>|\Phi_F^+|+|\Phi_{S\setminus F}^+|,
\]
a contradiction. Thus \eqref{eq:homology} produces at least two positive-degree homology classes when $F\subsetneq S$, which is incompatible with the homology of a sphere. Hence the only minimal nonface is $S$, so every proper subset of $S$ is a face and $\mathcal K=\partial\Delta^S$. This proves (i)$\Rightarrow$(iii).

Finally, (iii)$\Rightarrow$(ii) follows from the Castellana--Kitchloo identification with the unit adjoint sphere \cite[Theorem~1.1]{CK}.
\end{proof}

\begin{remark}
The irreducibility hypothesis is essential for this simple rigidity statement: in a product Coxeter system the equality case in the length argument can split across components.
\end{remark}

\section{Familiar specializations and applications}\label{sec:specializations}
The main formula becomes especially concrete in two classes familiar from topological combinatorics.  We first recover an ordinary polyhedral product from a product of rank-one groups, and then specialize to matroid independence complexes.
\subsection{Products of rank-one groups}
Let $G=(SU(2))^r$ and identify $S$ with $\{1,\ldots,r\}$. For $I\subseteq S$ one has
\[
  G/G_I\cong\prod_{j\notin I}S^2,
\]
and the map for $I\subset J$ is the projection that forgets the factors indexed by $J\setminus I$.

\begin{theorem}[Polyhedral-product identification]\label{thm:polyprod}
For every simplicial complex $\mathcal K\subseteq2^S$ there is a natural homotopy equivalence
\[
  \XK((SU(2))^r)\simeq(D^3,S^2)^{\mathcal K}.
\]
\end{theorem}

\begin{proof}
Let $D(I)=\prod_{j\notin I}S^2$ with projection maps, so that $\XK((SU(2))^r)=\hocolim D$, and let $E(I)=\prod_{j\in I}CS^2\times\prod_{j\notin I}S^2$.
For $I\subset J$ use the standard inclusions $S^2\hookrightarrow CS^2$ in the new cone coordinates. Collapsing each cone factor to its cone point defines an objectwise homotopy equivalence of diagrams $E\to D$, and the squares commute strictly. Homotopy invariance of the Bousfield--Kan construction therefore gives $\hocolim E\simeq\hocolim D$ \cite[Chapter~XII]{BK}.

The diagram $E$ is a diagram of CW subcomplex inclusions. More precisely, for each face $I$ the latching union $\bigcup_{J\subsetneq I}E(J)\subseteq E(I)$ is the subcomplex in which at least one cone coordinate indexed by $I$ lies in the base $S^2\subset CS^2$. It is therefore a CW subcomplex, and its inclusion is a closed cofibration. Hence the Projection Lemma \cite[Lemma~4.5]{WZZ} identifies $\hocolim E$ with $\operatorname{colim}E$. That colimit is precisely the polyhedral product $(CS^2,S^2)^{\mathcal K}=(D^3,S^2)^{\mathcal K}$ \cite{BBCG}.
\end{proof}

In this case $W\cong(\Z/2)^r$: every $D\subseteq S$ is the descent set of a unique element $w_D$ and $\ell(w_D)=|D|$. Formula \eqref{eq:homology} therefore becomes the familiar induced-subcomplex grading for this polyhedral product, with suspension shift $2|D|+1$.

\subsection{Matroids and the Tutte polynomial}
Let $M$ be a matroid on $S$ and put $\mathcal K=\Ind(M)$. Then $\mathcal K_D=\Ind(M|D)$. The independence complex of a coloopless rank-$r$ matroid is homotopy equivalent to a wedge of $T_M(0,1)$ spheres of dimension $r-1$; if a coloop is present, the complex is contractible \cite{BjornerMatroid,ACS}.

\begin{corollary}[Coxeter--Tutte formula]\label{cor:tutte}
Let $\bbk$ be a field. With $r_M(D)$ denoting the rank of $D$ in $M$,
\[
  P_{X_{\Ind(M)}(G)}(t;\bbk)
  =\sum_{w\in W}
  T_{M|\DesR(w)}(0,1)
  t^{2\ell(w)+r_M(\DesR(w))}.
\]
The identity term is $1$, since the empty matroid has $T(0,1)=1$ and rank zero.
\end{corollary}

\begin{proof}
Apply \eqref{eq:homology} to the restrictions $M|D$ and use the homotopy type of their independence complexes. If $M|D$ has a coloop, both sides contribute zero for that $D$ because $T_{M|D}(0,1)=0$.
\end{proof}

\begin{remark}[Type $A$]
For $W=S_{r+1}$, $\ell(w)$ is the inversion number and $\DesR(w)$ is the ordinary permutation descent set. Thus Corollary~\ref{cor:tutte} mixes inversion number, descent sets, matroid restrictions, ranks, and Tutte evaluations in one Poincar\'e-series formula.
\end{remark}

\section{The boundary simplex: adjoint spheres and bivectors}\label{sec:bivector}
We finally return to the example that motivated the construction.  For the boundary simplex the induced-subcomplex formula collapses to two homology classes, and the full chain model gives an explicit integral representative of the top class.
Return to a compact simply connected simple Lie group $G$ and specialize to $\mathcal K=\partial\Delta^S$. By Corollary~\ref{cor:schubertbar}, $\BK(W)$ is the Schubert cellular bar complex of the proper-parabolic diagram. Castellana and Kitchloo identify the corresponding Bousfield--Kan realization with $S(\mathfrak g)$ \cite[Theorem~1.1]{CK}. Since $\dim G=r+2\ell(w_0)$, Theorem~\ref{thm:perfect} gives the expected two-cell chain model in degrees $0$ and $\dim G-1$. The resulting higher-rank model complements the chain-and-diagram treatment of low-dimensional bivector orbit decompositions in \cite{companion}.

\begin{theorem}[Compatible proper-parabolic chain model]\label{thm:propermodel}
For every compact simply connected simple $G$, the proper-parabolic face diagram admits simultaneous Schubert cellular structures for which all projection maps are strict cellular chain maps. The resulting normalized bar complex has the descent decomposition \eqref{eq:boundarysplit} and the perfect reduction \eqref{eq:morse}.
\end{theorem}

\begin{proof}
Combine Proposition~\ref{prop:projection}, Corollary~\ref{cor:schubertbar}, Theorem~\ref{thm:exactsplit}, and Theorem~\ref{thm:perfect}.
\end{proof}

\begin{corollary}[Fundamental class of the adjoint sphere]
Under the Castellana--Kitchloo homeomorphism, the class of $\Omega_W$ in \eqref{eq:topcycle} maps, up to orientation sign, to the fundamental class of $S(\mathfrak g)$.
\end{corollary}

\begin{proof}
The class is a generator of the unique top integral homology group by Corollary~\ref{cor:topcycle}, and the realization is an oriented sphere.
\end{proof}

\begin{corollary}[Bivector spheres for $\operatorname{Spin}(n)$]\label{cor:spin}
For $n\ge6$, take $G=\operatorname{Spin}(n)$ and use the standard identifications $\mathfrak{spin}(n)\cong\mathfrak{so}(n)\cong\Lambda^2\R^n$ \cite[Chapter~II]{LM}. Then the complete proper-parabolic face diagram of the unit bivector sphere $S(\Lambda^2\R^n)=S^{\binom n2-1}$ has a strict Schubert cellular bar complex with a perfect integral Morse reduction to two critical generators. The reduction is controlled by right descent sets in the Weyl group of type $B_{(n-1)/2}$ for odd $n$ and $D_{n/2}$ for even $n$, with $D_3\cong A_3$ when $n=6$ \cite[Chapters~2--3]{Humphreys}.
\end{corollary}

\section{Concluding remarks}
The functor $\mathcal K\mapsto\XK(G)$ places the proper-parabolic adjoint sphere in a larger family of parabolic homotopy colimits. Its normalized bar differential decomposes by Weyl-group element, and the $w$-summand reduces integrally to a suspension of the augmented chains of $\mathcal K_{\DesR(w)}$. Over a field, Hochster's formula identifies these summands with squarefree multigraded Betti pieces of $\bbk[\mathcal K]$. Functoriality and simple connectivity give the homotopy-detection theorem, while complementary descent sets and minimal nonfaces give the Alexander symmetry and sphere rigidity results. The rank-one product and matroid cases connect the same decomposition with polyhedral products and Tutte polynomials.

The results here are additive. Natural further questions concern the cup product on $\XK(G)$, equivariant refinements, and filtrations of Coxeter complexes suggested by the descent summands. These questions require information beyond the cellular differential studied here.

\appendix
\section{Cell enumeration for the boundary simplex}\label{sec:enumerator}
Return to the boundary-simplex indexing complex. Let
\[
  \Phi_m(z)=\sum_{k=1}^m k!\begin{Bmatrix}m\\k\end{Bmatrix}z^{k-1},
  \qquad \Phi_0(z)=0,
\]
where $\begin{Bmatrix}m\\k\end{Bmatrix}$ is a Stirling number of the second kind, and define
\[
  \Psi_{r,d}(z)=\sum_{j=0}^{r-d}\binom{r-d}{j}\Phi_{d+j}(z).
\]
\begin{proposition}[Bigraded cell enumerator]\label{prop:enumerator}
Let
\[
  F_W(q,z)=\sum_{w\in W}\sum_{a\ge0}
  \#\{\text{bar cells labelled by }w\text{ of bar degree }a\}
  q^{\ell(w)}z^a.
\]
Then
\begin{equation}\label{eq:enumerator}
  F_W(q,z)=\sum_{w\in W}q^{\ell(w)}\Psi_{r,\desR(w)}(z).
\end{equation}
Equivalently,
\[
  F_W(q,z)=\sum_{I\subsetneq S}\left(\sum_{w\in W^I}q^{\ell(w)}\right)\Phi_{r-|I|}(z).
\]
\end{proposition}

\begin{proof}
For fixed $w$, the bar cells are the relative simplices of $(K_S,L_{\DesR(w)})$. A strict chain of bar degree $a$ beginning at a fixed $I_0$ is equivalent to an ordered partition of the $r-|I_0|$ elements outside $I_0$ into $a+1$ nonempty blocks, the last block being the complement of the terminal proper subset. This gives the factor $\Phi_{r-|I_0|}(z)$. Summing over $I_0\subseteq S\setminus\DesR(w)$ and grouping by the number of omitted elements gives $\Psi_{r,\desR(w)}(z)$. Reversing the two finite sums gives the parabolic form.
\end{proof}

\begin{corollary}[Descent-wise cancellation polynomials]\label{cor:cancellationpoly}
For every $0\le d\le r$ there is a polynomial $\Theta_{r,d}(z)\in\mathbb N[z]$ such that
\[
\Psi_{r,d}(z)=
\begin{cases}
1+(1+z)\Theta_{r,0}(z),&d=0,\\
(1+z)\Theta_{r,d}(z),&0<d<r,\\
z^{r-1}+(1+z)\Theta_{r,r}(z),&d=r.
\end{cases}
\]
The coefficients of $\Theta_{r,d}$ count matched unit pairs by lower bar degree for the matching of Lemma~\ref{lem:booleanmatching}.
\end{corollary}

\begin{proof}
Apply the cell-count identity to the Morse matching in each relative Boolean summand. Each matched pair contributes one cell in adjacent bar degrees, hence a factor $1+z$, while the critical cells are exactly those listed in Lemma~\ref{lem:booleanmatching}.
\end{proof}

\section*{Acknowledgements}
This work was supported by the REFRESH project--Research Excellence For REgion Sustainability and High-tech Industries, project No.~CZ.10.03.01/00/22\_003/0000048, via the Operational Programme Just Transition, and by the University of Ostrava, Grant No.~SGS05/P\v RF/2026.


\begin{thebibliography}{99}
\bibitem{AR} A.~Adams and V.~Reiner, A colorful Hochster formula and universal parameters for face rings, \emph{J. Commut. Algebra} 15 (2023), no.~2, 151--176. DOI: 10.1216/jca.2023.15.151.

\bibitem{ACS} F.~Ardila, F.~Castillo, and J.~A.~Samper, The topology of the external activity complex of a matroid, \emph{Electron. J. Combin.} 23 (2016), no.~3, Paper P3.8. DOI: 10.37236/5042.

\bibitem{BBCG} A.~Bahri, M.~Bendersky, F.~R.~Cohen, and S.~Gitler, The polyhedral product functor: a method of computation for moment-angle complexes, arrangements and related spaces, \emph{Adv. Math.} 225 (2010), 1634--1668. DOI: 10.1016/j.aim.2010.03.026. Preprint: arXiv:0711.4689.

\bibitem{BjornerMatroid} A.~Bj\"orner, The homology and shellability of matroids and geometric lattices, in \emph{Matroid Applications}, N.~White (ed.), Encyclopedia of Mathematics and its Applications 40, Cambridge University Press, Cambridge, 1992, 226--283.

\bibitem{BB} A.~Bj\"orner and F.~Brenti, \emph{Combinatorics of Coxeter Groups}, Graduate Texts in Mathematics 231, Springer, New York, 2005.

\bibitem{BK} A.~K.~Bousfield and D.~M.~Kan, \emph{Homotopy Limits, Completions and Localizations}, Lecture Notes in Mathematics 304, Springer, Berlin, 1972.

\bibitem{Brion} M.~Brion, Lectures on the geometry of flag varieties, in \emph{Topics in Cohomological Studies of Algebraic Varieties}, Trends in Mathematics, Birkh\"auser, 2005, 33--85. DOI: 10.1007/3-7643-7342-3\_2.

\bibitem{CK} N.~Castellana and N.~Kitchloo, A homotopy construction of the adjoint representation for Lie groups, \emph{Math. Proc. Cambridge Philos. Soc.} 133 (2002), no.~3, 399--409. DOI: 10.1017/S0305004102005947.

\bibitem{Chari} M.~K.~Chari, On discrete Morse functions and combinatorial decompositions, \emph{Discrete Math.} 217 (2000), no.~1--3, 101--113. DOI: 10.1016/S0012-365X(99)00258-7.

\bibitem{CCK} S.~Cho, S.~Choi, and S.~Kaji, Geometric representations of finite groups on real toric spaces, \emph{J. Korean Math. Soc.} 56 (2019), no.~5, 1265--1283. DOI: 10.4134/JKMS.j180646.

\bibitem{DJV} T.~Douvropoulos and M.~Josuat--Verg\`es, Homology character of the parabolic coset poset, arXiv:2509.11905, 2025.

\bibitem{FanWang} F.~Fan and X.~Wang, On the cohomology of moment-angle complexes associated to Gorenstein* complexes, arXiv:1508.00159, 2015, revised 2016.

\bibitem{GS} J.~Grbi\'c and M.~Staniforth, Duality in toric topology, in \emph{Toric Topology and Polyhedral Products}, Fields Institute Communications 89, Springer, Cham, 2024, 137--147. DOI: 10.1007/978-3-031-57204-3\_8.

\bibitem{Hatcher} A.~Hatcher, \emph{Algebraic Topology}, Cambridge University Press, Cambridge, 2002.

\bibitem{Hochster} M.~Hochster, Cohen--Macaulay rings, combinatorics, and simplicial complexes, in \emph{Ring Theory II (Proc. Second Oklahoma Conference)}, B.~R.~McDonald and R.~Morris (eds.), Lecture Notes in Pure and Applied Mathematics 26, Marcel Dekker, New York, 1977, 171--223.

\bibitem{Humphreys} J.~E.~Humphreys, \emph{Reflection Groups and Coxeter Groups}, Cambridge Studies in Advanced Mathematics 29, Cambridge University Press, Cambridge, 1990.

\bibitem{JW} M.~J\"ollenbeck and V.~Welker, \emph{Minimal Resolutions via Algebraic Discrete Morse Theory}, Memoirs of the American Mathematical Society 197 (2009), no.~923. DOI: 10.1090/memo/0923.

\bibitem{KT} D.~Kishimoto and M.~Takeda, Torsion in the space of commuting elements in a Lie group, \emph{Canad. J. Math.} 76 (2024), no.~3, 1033--1061. DOI: 10.4153/S0008414X23000317.

\bibitem{Kitchloo} N.~Kitchloo, On the topology of Kac--Moody groups, \emph{Math. Z.} 276 (2014), no.~3--4, 727--756. DOI: 10.1007/s00209-013-1220-3.

\bibitem{KozlovCollapse} D.~N.~Kozlov, Collapsing along monotone poset maps, \emph{Int. J. Math. Math. Sci.} 2006, Article ID 79858. DOI: 10.1155/IJMMS/2006/79858.

\bibitem{KozlovMorse} D.~N.~Kozlov, Discrete Morse theory for free chain complexes, \emph{C. R. Math. Acad. Sci. Paris} 340 (2005), no.~12, 867--872. DOI: 10.1016/j.crma.2005.04.036.

\bibitem{LM} H.~B.~Lawson, Jr. and M.-L.~Michelsohn, \emph{Spin Geometry}, Princeton Mathematical Series 38, Princeton University Press, Princeton, NJ, 1989.

\bibitem{LS} I.~Yu.~Limonchenko and G.~D.~Solomadin, On the homotopy decomposition for the quotient of a moment-angle complex and its applications, \emph{Proc. Steklov Inst. Math.} 317 (2022), 117--140. DOI: 10.1134/S0081543822020067.

\bibitem{PRV} T.~Panov, N.~Ray, and R.~Vogt, Colimits, Stanley--Reisner algebras, and loop spaces, in \emph{Categorical Decomposition Techniques in Algebraic Topology}, Progress in Mathematics 215, Birkh\"auser, Basel, 2004, 261--291. DOI: 10.1007/978-3-0348-7863-0\_15.

\bibitem{Petersen} T.~K.~Petersen, A two-sided analogue of the Coxeter complex, \emph{Electron. J. Combin.} 25 (2018), no.~4, Paper P4.64. DOI: 10.37236/8015.

\bibitem{companion} B.~Schneider, D.~Schneiderov\'a (Barseghyan), and Y.~Zhang, A chain- and diagram-level semantics for morphological calculus refinement, monodromy, and bivector orbit decompositions, arXiv:2608.11325 [math.AG], 2026.

\bibitem{WZZ} V.~Welker, G.~M.~Ziegler, and R.~T.~\v Zivaljevi\'c, Homotopy colimits---comparison lemmas for combinatorial applications, \emph{J. Reine Angew. Math.} 509 (1999), 117--149. DOI: 10.1515/crll.1999.509.117.
\end{thebibliography}
\end{document}